\documentclass[12pt]{article}
\usepackage{amsmath,enumitem}
\usepackage[latin1]{inputenc}
\usepackage{amssymb}
\usepackage{amsthm}
\usepackage{xcolor}
\usepackage{hyperref}

\theoremstyle{plain}
\newtheorem{theorem}{Theorem}
\newtheorem{corollary}[theorem]{Corollary}

\newtheorem{observation}[theorem]{Observation}
\theoremstyle{definition}
\newtheorem{definition}[theorem]{Definition}

\def\li{\text{\rm Li}}

\def\stirling#1#2{\genfrac{\{}{\}}{0pt}{}{#1}{#2}}
\def\euler#1#2{\genfrac{\langle}{\rangle}{0pt}{}{#1}{#2}}

\newcommand{\seqnum}[1]{\href{http://oeis.org/#1}{\underline{#1}}}

\date{\today }

\title{Poly-Bernoulli Numbers and Eulerian Numbers}

\author{Be\'ata B\'enyi\\
Faculty of Water Sciences\\
National University of Public Service\\
Hungary\\
\tt beata.benyi@gmail.com\\
\and
P\'eter Hajnal\\
 Bolyai Institute\\ University of Szeged\\ Hungary\\ and \\
 Alfr\'ed R\'enyi Institute of Mathematics\\
            Hungarian Academy of Sciences\\ Hungary\\
\tt hajnal@math.u-szeged.hu}

\begin{document}

\maketitle
\begin{abstract} 
In this note we prove combinatorially some new formulas 
connecting poly-Bernoulli numbers with negative indices
to Eulerian numbers.
\end{abstract}

\section{Introduction}

Kaneko \cite{Kaneko} introduced Poly-Bernoulli numbers 
\seqnum{A099594} during his investigations 
on multiple zeta values. He defined these numbers
by their generating function:
\begin{align}
\sum_{n=0}^{\infty} B_{n}^{(k)}\frac{x^n}{n!}=
\frac{\li _k(1-e^{-x})}{1-e^{-x}},
\end{align}
where \[\li_k(z)=\sum_{i=1}^{\infty}\frac{z^i}{i^k}\]
is the classical polylogarithmic function.
As the name indicates, poly-Bernoulli numbers 
are generalizations of the Bernoulli numbers. 
For $k=1$ $B_n^{(1)}$ are  the classical 
Bernoulli numbers with $B_{1}=\frac{1}{2}$. 
For negative $k$-indices poly-Bernoulli numbers are integers 
(see the values for small $n$, $k$ in Table 1) 
and have interesting combinatorial properties. 

\begin{table}[h]
\begin{center}
\begin{tabular}{|c||c|c|c|c|c|c|}
\hline
$n$\textbackslash $k$
  & 0 & 1 & 2 & 3 & 4 & 5\\
\hline\hline
0 & 1 & 1& 1& 1 & 1 & 1\\
\hline
1 & 1 & 2 & 4 & 8 & 16 & 32\\
\hline
2 & 1 & 4 & 14 & 46 & 146 & 454 \\
\hline
3 & 1 & 8 & 46 & 230 & 1066 & 4718 \\
\hline
4 & 1 & 16 & 146 & 1066 & 6906 & 41506\\
\hline
5 & 1 & 32 & 454 & 4718 & 41506 & 329462\\
\hline
\end{tabular}
\end{center}
\caption{The poly-Bernoulli numbers $B_n^{(-k)}$}\label{tabel}
\end{table}

Poly-Bernoulli numbers enumerate several combinatorial objects arisen
in different research areas, as for instance lonesum matrices,
$\Gamma$-free matrices, acyclic orientations of complete bipartite
graphs, alternative tableaux with rectangular shape, permutations with
restriction on the distance between positions and values, permutations
with excedance set $[k]$ etc. 
In \cite{BH1,BH2,BB}
the authors summarize in an actual list
the known interpretations, present connecting bijections and give
further references.

In this note we are concerned only with poly-Bernoulli numbers with
negative indices. For convenience, we denote by $B_{n,k}$ the
poly-Bernoulli numbers $B_n^{(-k)}$.

Kaneko derived two formulas for the poly-Bernoulli numbers with
negative indices: a formula that we call basic formula, and an
inclusion-exclusion type formula. The basic formula is
\begin{align}\label{combform}
B_{n,k}=\sum_{m=0}^{\min(n,k)}(m!)^2{n+1\brace m+1}{k+1\brace m+1},
\end{align}
where ${n\brace k}$ denotes the Stirling numbers 
of the second kind \seqnum{A008277} that counts 
the number of partitions of an $n$-element 
set into $k$ non-empty blocks \cite{GKP}. 
The inclusion-exclusion type formula is
\begin{align}
B_{n,k}=\sum_{n=0}^{\infty}(-1)^{n+m}m!{n\brace m}(m+1)^k.
\end{align}

Kaneko's proofs were algebraic, based on manipulations
of generating functions.
The first combinatorial investigation of
poly-Bernoulli numbers was done by 
Brewbaker \cite{Brewbaker}. 
He defined $B_{n,k}$ as the number of lonesum 
matrices of size $n\times k$.
He proved both formulas combinatorially; hence,
he proved the equivalence of the algebraic definition
and the combinatorial one.

Bayad
and Hamahata \cite{Bayad}
introduced poly-Bernoulli polynomials by the following
generating function:
\begin{align*}
\sum_{n=0}^{\infty}B_n^{(k)}(x)\frac{t^n}{n!}=
\frac{\li_k(1-e^{-t})}{1-e^{-t}}e^{xt}.
\end{align*}
For negative indices the polylogarithmic 
function converges for $|z|<1$ and equals to
\begin{align}\label{lieu}
\li_{-k}(z)=\frac{\sum_{j=0}^{k} \euler{k}{j}z^{k-j}}{(1-z)^{k+1}},
\end{align}
where $\euler{k}{j}$ is the Eulerian 
number \cite{GKP} \seqnum{A008282} given for instance by:
\begin{align}\label{eulerszita}
 \euler{k}{j} =\sum_{i=0}^j(-1)^i\binom{k+1}{i}(j-i)^k.
\end{align}
In \cite{Bayad} the authors used analytical methods 
to show that for $k\leq 0$ it holds 
\begin{align}\label{pBpol}
B_{n}^{(k)}(x)=\sum_{j=0}^{|k|}\euler{|k|}{j}
\sum_{m=0}^{|k|-j}\binom{|k|-j}{m}(-1)^m(x+m-|k|-1)^n.
\end{align} 
The evaluation of \eqref{pBpol} at $x=0$ leads 
to a new explicit formula of the poly-Bernoulli numbers 
involving Eulerian numbers.
\begin{theorem}\cite{Bayad}\label{BayHam}
For all $k> 0$ and $n> 0$ it holds
\begin{align}\label{pBeu2}
B_{n,k}=\sum_{j=0}^{k}\left\langle k\atop j\right\rangle
\sum_{m=0}^{k-j}(-1)^m\binom{k-j}{m}(k+1-m)^n.
\end{align}
\end{theorem}

We see that the Eulerian numbers and the defining generating function
of poly-Bernoulli numbers for negative $k$ are strongly related.

In this note we prove this formula purely combinatorially. Moreover,
we show four further new formulas for poly-Bernoulli numbers
involving Eulerian numbers.

\section{Main results}

In our proofs a special class of permutations plays the key role. We
call this permutation class Callan permutations because Callan
introduced this class in \cite{Callan} as a combinatorial interpretation
of the poly-Bernoulli numbers.
We use the well-known notation: $[N]:=\{1,2,\ldots,N\}$.

\begin{definition}
\emph{Callan permutation} of $[n+k]$ is 
a permutation such that each substring 
whose support belongs to $N=\{1,2,\ldots ,n\}$ or  
$K=\{n+1,n+2,\ldots n+k\}$ is increasing.
\end{definition}

Let $\mathcal{C}_n^k $ denote the set of Callan permutations of
$[n+k]$. We call the elements in $N$ the \emph{left-value elements}
and the elements in $K$ the \emph{right-value elements}.  For
instance, for $n=2, k=2$ the Callan permutations are (we write the
left-value elements in red, right-value elements in blue):
\[
{\color{red}{1}}{\color{red}{2}}{\color{blue}{3}}{\color{blue}{4}},
{\color{red}{1}}{\color{blue}{3}}{\color{red}{2}}{\color{blue}{4}},
{\color{red}{1}}{\color{blue}{4}}{\color{red}{2}}{\color{blue}{3}},
{\color{red}{1}}{\color{blue}{3}}{\color{blue}{4}}{\color{red}{2}},
{\color{red}{2}}{\color{blue}{3}}{\color{red}{1}}{\color{blue}{4}},
{\color{red}{2}}{\color{blue}{4}}{\color{red}{1}}{\color{blue}{3}},
{\color{red}{2}}{\color{blue}{3}}{\color{blue}{4}}{\color{red}{1}},
{\color{blue}{3}}{\color{red}{1}}{\color{red}{2}}{\color{blue}{4}},
{\color{blue}{3}}{\color{red}{1}}{\color{blue}{4}}{\color{red}{2}},
{\color{blue}{3}}{\color{red}{2}}{\color{blue}{4}}{\color{red}{1}},
{\color{blue}{3}}{\color{blue}{4}}{\color{red}{1}}{\color{red}{2}},
{\color{blue}{4}}{\color{red}{1}}{\color{red}{2}}{\color{blue}{3}},
{\color{blue}{4}}{\color{red}{1}}{\color{blue}{3}}{\color{red}{2}},
{\color{blue}{4}}{\color{red}{2}}{\color{blue}{3}}{\color{red}{1}}.
\]
It is easy to see that Callan permutations are enumerated by the
poly-Bernoulli numbers, but for the sake of completeness, we recall
the sketch of the proof of this theorem.

\begin{theorem}\cite{Callan}
\[
|\mathcal{C}_n^k|=\sum_{m=0}^{\min(n,k)}(m!)^2
{n+1\brace m+1}{k+1\brace m+1}=B_{n,k}.
\]
\end{theorem}

\begin{proof}(Sketch)
We extend our universe with $\textcolor{red}{0}$, a special left-value element and  
with $\textcolor{blue}{n+k+1}$, a special right-value element. $\widehat{N}=N\cup\{\textcolor{red}{0}\}$  and 
$\widehat{K}=K\cup {(\textcolor{blue}{n+k+1})}$.
Let $\pi\in\mathcal C_n^{k}$.
Let $\widetilde\pi=\textcolor{red}{0}\pi(\textcolor{blue}{n+k+1})$.
Divide $\widetilde\pi$ into maximal blocks of consecutive elements in 
such a way that each block is a subset of $\widehat{N}$
(\emph{left blocks})
or a subset of $\widehat{K}$ (\emph{right blocks}).
The partition starts with a left block
(the block of $\textcolor{red}{0}$) and ends with a right block
(the block of $(\textcolor{blue}{n+k+1})$).
So the left and right blocks alternate, and their number is the same,
say $m+1$.

Describing a Callan permutation is 
equivalent to
specifying $m$, a partition $\Pi_{\widehat N}$ of $\widehat N$
into $m+1$ classes (one class is the class
of $\textcolor{red}{0}$, the other $m$ classes are called \emph{ordinary
classes}), a partition $\Pi_{\widehat K}$  of $\widehat K$
into $m+1$ classes
($m$ many of them not containing $(\textcolor{blue}{n+k+1})$,
these are the ordinary classes), and two orderings of the ordinary classes.
After specifying the components, we
need to merge the two ordered set of classes (starting with the
nonordinary class of $\widehat{N}$ and ending with the nonordinary class of $\widehat{K}$),
and list the elements of classes in increasing order.
The classes of our partitions will form  
the blocks of the Callan permutations.
We will refer to the blocks coming from ordinary classes
as \emph{ordinary blocks}.

This proves the claim.
\end{proof}

The main results of this note are the next five formulas for the
number of Callan permutations and hence, for the poly-Bernoulli
numbers. We present elementary combinatorial proofs of the theorems in
the next section. Theorem \ref{theo2} is equivalent to theorem \ref{BayHam}; we recall the theorem in the combinatorial setting.

\begin{theorem}\label{theo0}
For all $k> 0$ and $n> 0$ it holds
\begin{align}\label{pBeupar}
|\mathcal{C}_n^k|=
\sum_{m=0}^{\min{(n,k)}}\sum_{i=0}^n\sum_{j=0}^k 
\euler{n}{i}\euler{k}{j}\binom{n+1-i}{m+1-i}\binom{k+1-j}{m+1-j}=
B_{n,k}.
\end{align}
\end{theorem}

\begin{theorem}\label{theo3}
For all $k> 0$ and $n>0$ it holds
\begin{align}\label{pBeu3}
|\mathcal{C}_n^k|=\sum_{j=0}^{k}
\left\langle k\atop j\right\rangle
\sum_{m=0}^{k+2-j}\binom{k+2-j}{m}(m+j-1)!{n\brace m+j-1}=B_{n,k}.
\end{align}
\end{theorem}

\begin{theorem}\label{theo1}
For all $k> 0$ and $n>0$ it holds
\begin{align}\label{pBeu1}
|\mathcal{C}_n^k|=\sum_{j=0}^{k}
\left\langle k\atop j\right\rangle
\sum_{m=0}^{j-1}(-1)^m\binom{j-1}{m}(k+1-m)^n=B_{n,k}.
\end{align}
\end{theorem}

\begin{theorem}\label{theo4}
For all $k> 0$ and $n>0$ it holds
\begin{align}\label{pBeu4}
|\mathcal{C}_n^k|=\sum_{j=0}^{k}
\left\langle k\atop j\right\rangle
\sum_{m=0}^{j+1}\binom{j+1}{m}(m+k-j)!{n\brace m+k-j}=B_{n,k}.
\end{align}
\end{theorem}

\begin{theorem}\label{theo2}\cite{Bayad}
For all $k> 0$ and $n> 0$ it holds
\begin{align}\label{pBeu2}
|\mathcal{C}_n^k|=\sum_{j=0}^{k}\left\langle k\atop j\right\rangle
\sum_{m=0}^{k-j}(-1)^m\binom{k-j}{m}(k+1-m)^n=B_{n,k}.
\end{align}
\end{theorem}

\section{Proofs of the main results}

Eulerian numbers play the crucial role in these formulas. Though
Eulerian numbers are well-known, we think it could be helpful for
readers who are not so familiar with this topic to recall some basic
combinatorial properties.

\subsection{Eulerian numbers}

First we need some definitions and notation. 
Let $\pi=\pi_1\pi_2\ldots \pi_n$ 
be a permutation of $[n]$. We call $i\in [n-1]$ 
a \emph{descent} 
(resp.~\emph{ascent}) of $\pi$ if $\pi_i>\pi_{i+1}$ 
(resp.~$\pi_i<\pi_{i+1}$). 
Let $D(\pi)$ (resp.~$A(\pi)$) denote the set of descents 
(resp.~ the set of ascents) of the permutation $\pi$. 
For instance, $\pi=361487925$ has 3 descents and 
$D(\pi)= \{2, 5, 7\}$, 
while it has $5$ ascents and $A(\pi)=\{1,3,4,6,8\}$.

Eulerian numbers $\euler{k}{j}$ counts the permutations of 
$[k]$ with $j-1$ descents. A permutation $\pi\in S_n$ with $j-1$ descents 
is the union of $j$ increasing subsequences of consecutive entries, 
so called \emph{ascending runs}. So, in other words 
$\left\langle k\atop j\right\rangle$ is the number of permutations of $[k]$ 
with $j$ ascending runs. In our example, $\pi$ is 
the union of 4 ascending runs: $36$, $148$, $79$, and $25$. 

There are several identities involving Eulerian numbers, 
see for instance \cite{Bona}, \cite{GKP}. 
We will use a strong connection 
between the surjections/ordered partitions and Eulerian numbers:

\begin{align}\label{opeu}
r!{k\brace r}=\sum_{j=0}^r\euler{k}{j}\binom{k-j}{r-j}.
\end{align}

\begin{proof}
We take all the partitions of $[k]$
into $r$ classes. Order the classes,
and list the elements in increasing order.
This way we obtain permutations of $[k]$.
Counting by multiplicity we get
$r!{k\brace r}$ permutations.
All of these have at most $r$ ascending runs.

Take a permutation with $j(\leq r)$ ascending runs. 
How many times did we list it in the previous paragraph?
We split the ascending runs 
by choosing $r-j$ places out of the $k-j$ ascents 
to obtain all the initial $r$ blocks.
The multiplicity is $\binom{k-j}{r-j}$.
This proves our claim.
\end{proof}

Inverting \eqref{opeu} gives immediately,
\begin{align*}
\euler{k}{j}= \sum_{r=1}^j(-1)^{j-r}r!{k\brace r} \binom{k-r}{j-r}.
\end{align*}
In the previous section we mentioned the close relation between Eulerian numbers and the polylogarithmic function $\li_k(x)$. 
Here we recall one possible derivation of the identity \eqref{lieu} 
following \cite{Bona}.
 
\begin{align*}
\sum_{j=0}^k \euler{k}{j} x^j&=\sum_{j=0}^\infty \euler{k}{j} x^j=
\sum_{j=0}^\infty
\sum_{i=0}^{j}(-1)^i\binom{k+1}{i}(j-i)^kx^j=\\
&=\sum_{j=1}^\infty \sum_{i=0}^j(-1)^{k-i}\binom{k+1}{j-i}i^kx^j
=\sum_{i=0}^\infty \sum_{j=i}^\infty(-1)^{k-i}\binom{k+1}{j-i}i^kx^j=\\
&=\sum_{i=0}^\infty i^kx^i
\left(\sum_{j=i}^\infty\binom{k+1}{j-i}(-x)^{j-i}\right)=
\sum_{i=0}^\infty i^kx^i
\left(1-x\right)^{k+1}=
\\
&=(1-x)^{k+1}\sum_{i=0}^\infty i^kx^i=(1-x)^{k+1}\li_{-k}(x).
\end{align*}
Plugging \eqref{eulerszita} for $\euler{k}{j}$, exchanging $i$ to
$j-i$, changing the order of the summation; and 
finally, applying the binomial theorem we get the
result.

\subsection{Combinatorial proofs of the theorems}

Now we turn our attention to the proofs of our theorems. For the sake
of convenience, thanks to our color coding
(left-value elements are red, and right-value elements
are blue), we rewrite the set of right-value elements as
$K=\{{\textcolor{blue}{1}}, {\color{blue}{2}}, 
\ldots,{\color{blue}{k}}\}$, and
$\widehat{K}=K\cup\{\textcolor{blue}{k+1}\}$. 
We can do this without
changing essentially Callan permutations, since we just need the
distinction between the elements $N$ and $K$ and an order in $N$ and
$K$. If we consider separately the left-value elements and right-value
elements in the permutation $\pi$ the elements of $N$ form a
permutation of $[n]$ and the elements of $K$ form a permutation of
$[k]$. We let $\pi^r$ denote the permutation restricted to the
right-value elements and we let $\pi^\ell$ denote the permutation
restricted to the left-value elements. For instance, for
\[
\pi=\color{red}{023}\color{blue}{145}\color{red}{47}
    \color{blue}{28}\color{red}{18}\color{blue}{3}
    \color{red}{569}\color{blue}{679},\]
$\pi^r={\textcolor{blue}{145283679}}$, while 
$\pi^\ell={\textcolor{red}{0234718569}}$.

\begin{proof}[Proof of Theorem 4.]
We consider the last entries of the blocks in the restricted 
permutations $\pi^\ell$ resp.~$\pi^r$. Some of the blocks end 
with a descent and some of the blocks not. 
(The special elements $\textcolor{red}{0}$ 
and $\textcolor{blue}{k+1}$ are 
neither descents nor ascents of the permutations.) 
Let $i$ be the number of ascending runs and $j$ 
the number of ascending runs in $\pi^r$. Let further $m$ be 
the number of ordinary blocks of both types. 
The $i-1$ descents of $\pi^\ell$ determine $i$
blockendings; hence, we are missing
$m-(i-1)$ further blockendings with ascents. 
Similarly, the $j-1$ descents of $\pi^r$ determine $j-1$ 
blockendings and there are 
further $m-(j-1)$ blockendings with an ascent. 

Given a pair $(\pi^\ell,\pi^r)$ 
with $|D(\pi^\ell)|=i-1$ and $|D(\pi^r)|=j-1$, 
we can construct a Callan permutations, according to the above arguments.
In our running example, 
$\pi^\ell={\textcolor{red}{0234718569}}$
and we need to choose $3-2=1$
blockendings from $9-2$ possibilities.
In general, we need to choose $m-(i-1)$
blockendings from $n-(i-1)$ possibilities. And analogously for $\pi^r$, we need to choose  $m-(j-1)$
blockendings from $k-(j-1)$ possibilities.
Hence, for a given a pair $(\pi^\ell,\pi^r)$ 
with $|D(\pi^\ell)|=i-1$ and $|D(\pi^r)|=j-1$
we have
\begin{align*}
%\sum_{i=0}^n\sum_{j=0}^k
\binom{n+1-i}{m+1-i}\binom{k+1-j}{m+1-j}
\end{align*} 
different corresponding Callan permutations.
Since the number of pairs $(\pi^\ell,\pi^r)$ with $|D(\pi^\ell)|=i-1$ 
and $|D(\pi^r)|=j-1$ 
is $\euler{n}{i}\euler{k}{j}$
Theorem 4 is proven. 
\end{proof}

Note that \eqref{pBeupar} is actually a rewriting of the basic
combinatorial formula \eqref{combform} in terms of Eulerian numbers
using the relation \eqref{opeu} between the number of ordered
partitions and Eulerian numbers.

Now we enumerate Callan permutations according to the number of
descents in $\pi^r$. Given a permutation $\pi^r$ with $j-1$ descents
we determine the number of ways to merge $\pi^r$ with left-value
elements to obtain a valid Callan permutation. Let
$D=\{d_1,d_2,\ldots, d_{j-1}\}$ be the set of descents of $\pi^r$. In
our running example, $j=3$ and $D=\{3,5\}$.  We code the positions of
the left-values comparing to the right-values by a word $w$. We let
$w_i$ be the number of right-values that are to the left of the
left-value $i$. In our example, $w_1=5$, since there are 5 right-value
elements preceding the left-value $\textcolor{red}{1}$, $w_2=0$,
because there are no right-value elements preceding the left-value
$\textcolor{red}{2}$, etc. Hence, $w=500366356$. Note that the blocks
of the left-value elements can be recognized from the word: The
positions $i$, for which the values $w_i$ are the same contain
the elements of a block. We call a word
\emph{valid respect to $\pi^r$} if the augmentation of $\pi^r$ according
to the word $w$ leads to a valid Callan permutation.

\begin{observation}
A word $w$ is valid respect to a permutation $\pi^r$ if and only if
it contains every value $d_i$ of the descent set of $\pi^r$.
\end{observation}

\begin{proof}
In a Callan permutation the substrings restricted to $K$ or $N$ 
are increasing subsequences. Given $\pi^r$ with descent set $D$, 
each $d_i\in D$ has to be the last element of a block 
in the ordered partition of the set K, the set of
right-value elements.
Hence, each right-value with position $d_i$ in $\pi^r$ has 
to be followed by a left-value element in the Callan permutation. 
In our word $w$ at the position of this
left-value element there is a $d_i$. 

For the converse, 
assume that our word $w$ contains at least one $d_i$,
for any $d_i\in D(\pi^r)$.
There is at least one left-value element
with $d_i$ in $w$ at its position.
This implies that if we combine $\pi^r$ and $\pi^\ell$
then in $\pi^r$ the position of the descent
will be interrupted by a left-value element. 
The combined permutation will be a Callan permutation.
\end{proof}

\begin{corollary}
The number of valid words respect to $\pi^r$ depends only on 
the number of descents in $\pi^r$.
\end{corollary}

We let $w^{j-1}$ denote a word that is valid to a $\pi^r$ with $j-1$
descents and we let $W(\pi^r)$ denote the set of words $w^{j-1}$.  The
number of Callan permutations of size $n+k$ is the number of
pairs $(\pi^r,w^{j-1})$, where $\pi^r$ is a permutation of $[k]$ with
$j-1$ descents and $w^{j-1}\in W(\pi^r)$.
We denote $|W(\pi^r)|$ by $w(j-1)$.
Hence,
\[
|\mathcal{C}_n^k|=\sum_{j=1}^k\euler{k}{j}w(j-1).
\]

The next two proofs are based on 
two different ways of determining $w(j-1)$,
i.e.~enumerating those $w^{j-1}$'s that are
valid to a $\pi^r$ with $j-1$
descents.

\begin{proof}[Proof of Theorem 5.]
Fix $\pi^r$ and take a $w^{j-1}\in W(\pi^r)$.
$w^{j-1}$ corresponds to an ordered
partition of $[n]$ into at least $j-1$ blocks. 
Let $j-1+m$ be the number
of the blocks. 

First, we take an ordered partition of 
$\color{red}{\{1,2\ldots,n\}}$ into $m+j-1$
non-empty blocks in $(m+j-1)!\stirling{n}{m+j-1}$ ways. 
Then we refine the partition
of $\pi^\ell$, defined by the descents.
For the refinement we need to choose
additional places for the $m$ blocks.
These places can be before the first element
of $\pi^\ell$,
or at an ascent.
We have $\binom{k+2-j}{m}$ choices.
This proves
\eqref{pBeu3}.
\end{proof}

\begin{proof}[Proof of Theorem 6.]
Now we calculate $w(j-1)$ using the
inclusion-exclusion principle. The total number of words of length
$n$ with entries $\{0,1,\ldots, k\}$ ($w_i=0$ if the left-value $i$
is in the first block of the Callan permutation) is $(k+1)^n$. We
have to reduce this number with the number of not valid words
respect to $\pi^r$, with the words that do not contain at least one
of the $d_i\in D$. Let $A_{s}$ be the set of words that do not contain
the value $s$. $|\overline{\cup_{s\in D} A_s}|$ is to be determined. Clearly,
$|A_s|=(k+1-1)^n$ and this number does not depend on the choice of
$s$; hence, we have $\sum_{s\in D}|A_s|=k^n(j-1)$. 
The $|A_s\cap A_t|=(k+1-2)^n$ and 
$\sum_{s,t\in D}|A_s\cap A_t|=(k-1)^n\binom{j-1}{2}$. Analogously,
$|\cap_{l=1}^{m}{A_{s_l}}|=(k+1-m)^n\binom{j-1}{m}$. The
inclusion-exclusion principle gives
\[
w(j-1)=\sum_{m=0}^{j-1}(-1)^m\binom{j-1}{m}(k+1-m)^n,
\]
and this implies \eqref{pBeu1}.  
\end{proof}

\begin{proof}[Proof of Theorem 7. and Theorem 8.]
The claims follow by the symmetry of Eulerian numbers. 
If we reverse a permutation of $[k]$ with $j-1$ descents we obtain a 
permutation with $k-(j-1)-1$ descents. 
According to our previous arguments a pair $(\pi^r, w^{k-j})$, 
where $\pi^r$ is a permutation with $k-j$ descents and $w^{k-j}$ is a
 valid word respect to $\pi^r$ determines a Callan permutation. Hence, 
\[
|\mathcal{C}_n^k|=\sum_{j=1}^{k}\euler{k}{k-j+1}w(k-j)=\sum_{j=1}^k\euler{k}{j}w(k-j).
\]
We have two formulas for $w(k-j)$ using the results of the proofs of the previous theorems.
\begin{align*}
w(k-j)&=\sum_{m=0}^{j+1}\binom{j+1}{m}(m+k-j)!\stirling{n}{m+k-j},\\
w(k-j)&=\sum_{m=0}^{k-j}(-1)^m\binom{k-j}{m}(k+1-m)^n.
\end{align*}
This implies \eqref{pBeu4} and \eqref{pBeu2}. 
\end{proof}

\bigskip\bigskip

2010 \textsl{Mathematics Subject Classification:} 05A05, 05A15, 
                                                  05A19, 11B83.\\
\textsl{Keywords:} Combinatorial identities, Eulerian number, 
                   poly-Bernoulli number.\\
(Concerned with sequences: \seqnum{A008282}, 
\seqnum{A027641}/\seqnum{A027642}, \seqnum{A008277}, 
\seqnum{A099594})

\end{document}